\documentclass[11pt]{article}

\usepackage{amssymb, amsmath, fullpage, amsthm, pictexwd, graphicx, subfigure, skak}
\usepackage[small]{caption}
\newtheorem{theorem}{Theorem}[section]
\newtheorem{lemma}[theorem]{Lemma}
\newtheorem{proposition}[theorem]{Proposition}
\newtheorem{corollary}[theorem]{Corollary}

\newtheorem{definition}[theorem]{Definition}

\newtheorem{remark}[theorem]{Remark}

\newcommand{\inte}{\operatorname{int}}
\newcommand{\BZ}{{\mathbb Z}}

\newcommand{\Pn}{\Psi_{a_n, \ldots , a_1}}
\newcommand{\Pk}{\Psi_{a_{k}, \ldots , a_1}}
\newcommand{\Pko}{\Psi_{a_{k+1}, \ldots , a_1}}
\newcommand{\St}{\operatorname{Stir}}
\newcommand{\lk}{\operatorname{lk}}
\newcommand{\del}{\operatorname{del}}
\newcommand{\CR}{\mathcal{R}}

\makeatletter
\@addtoreset{equation}{section}
\makeatother

\begin{document}
\title{Simplicial Complexes of Triangular Ferrers Boards}
\author{Eric Clark and Matthew Zeckner}
\maketitle

\begin{abstract}
We study the simplicial complex that arises from non-attacking
rook placements on a subclass of Ferrers boards that have $a_i$ rows of length
$i$ where $a_i>0$ and $i\leq n$ for some positive integer $n$.
In particular, we will investigate enumerative properties
of their facets, their homotopy type, and homology.
\end{abstract}

\section{Introduction}

A chessboard complex is the collection of all non-attacking rook positions on an $m \times n$ chessboard.  It is clear that
this is a simplicial complex as the removal of one rook from an admissible rook placement yields another admissible rook placement.
Notice that a placement of $i+1$ rooks corresponds to a simplex of dimension $i$.

In this paper, we will be studying the topology of the simplicial complex that arises from non-attacking rook placements
on triangular boards.
The triangular board $\Pn$ is a left justified board with $a_i>0$
rows of length $i$ for $1\leq i \leq n$.  In other words, given a positive integer $n$, the triangular
board $\Pn$ is the Ferrers board associated with the partition
$\pi = (n^{a_n},\ldots,1^{a_1})$ with $a_i>0$;
see Figure~\ref{figure_tri_board}.  The squares of the triangular board will be labeled $(i,j)$ for $i\leq j$
where $i$ represents the columns (numbered left to right) while $j$ represents the
rows (labeled bottom to top).  Motivated by results obtained using the Macaulay2 software package found in~\cite{Cook},
we begin by showing that for $a_i$ large enough, $\Sigma(\Psi_{a_n,\ldots,a_1}$),
the simplicial complex associated with rook placements on $\Psi_{a_n,\ldots,a_1}$, is a pure complex
that is vertex decomposable.

Next, we study the other extremal case.  The Stirling complex $\St(n)$, originally
defined by Ehrenborg and Hetyei~\cite{EhrenborgHetyei}, is the simplicial
complex formed by rook placements on the board $\Psi_{1,1,\ldots,1}$ with
$n$ rows; see Figure~\ref{figure_rook_placement}.
It is known that the $f$-vector of $\St(n)$ is given by
$f_i = S(n+1, n+1-i)$ for $i= 1, \ldots, (n-1)$ where $S(n,i)$ denotes the Stirling
number of the second kind; see~\cite[Corollary 2.4.2]{StanleyVol1}.
This complex is not pure.  We begin the study of $\St(n)$ by enumerating its
facets via generating functions and then use Discrete Morse theory to study its topology.

Chessboard complexes first appeared
in the 1979 thesis of Garst~\cite{Garst} concerning
Tits coset complexes.
By setting $G = \mathfrak{S}_n$ and $G_i = \{\sigma | \sigma (i) = i \}$ for
$i = 1, \ldots , m  \le n$, Garst obtained the chessboard complex
$M_{m,n} = \Delta (G;G_1, \ldots ,G_m)$.
Here, $\Delta (G;G_1, \ldots ,G_m)$ is the simplicial complex whose vertices
are the cosets of the subgroups
and whose facets have the form $\{gG_1, \ldots , gG_m \}$, for $g \in G$.

The chessboard complex later appeared in a paper by
Bj\"orner, Lov\'asz, Vre\'cica, and \v{Z}ivaljevi\'c~\cite{Bjorner_et_al}
where they gave a bound on the connectivity of the chessboard complex and
conjectured that their bound was sharp.
This conjecture was shown to be true by Shareshian and Wachs~\cite{Shar_Wach}.
In that same paper, Shareshian and Wachs also showed that if
the chessboard met certain criteria, then it contained torsion in its homology.

\section{Topological Tools}

For an introduction to combinatorial topology, basic definitions, and results,
we refer the reader to the books by Jonsson~\cite{JonssonBook} and Kozlov~\cite{KozlovBook}.
\begin{definition}
\emph{For a family $\Delta$ of sets and a set $\sigma$ of $\Delta$, the \emph{link} $\lk_\Delta(\sigma$) is the family of all $\tau \in \Delta$ such that
$\tau \cap \sigma = \emptyset$, and $\tau \cup \sigma \in \Delta$.  The \emph{deletion} $\del_\Delta (\sigma)$ is the family
of all $\tau \in \Delta$ such that $\sigma \cap \tau = \emptyset$.}
\end{definition}

\begin{definition}
\emph{A simplicial complex $\Delta$ is \emph{vertex decomposable} if
\begin{enumerate}
\item
Every simplex (including $\emptyset$ and $\{\emptyset\}$) is vertex decomposable.
\item
$\Delta$ is pure and contains a $0$-cell $x$ -- a \emph{shedding vertex} -- such that $\del_\Delta(x)$ and $\lk_\Delta(x)$
are both vertex decomposable.
\end{enumerate}}
\end{definition}

Showing that a simplicial complex is vertex decomposable is useful in determining
the topology of the complex as can be seen in the following theorem.

\begin{theorem}{\cite[Theorems 3.33 and 3.35]{JonssonBook}}
Let $\Delta$ be a simplicial complex of dimension $d$.  If the complex $\Delta$ is vertex
decomposable, then $\Delta$ is homotopy equivalent to a wedge of spheres of
dimension $d$.
Moreover, we have the following implications: \\
Vertex Decomposable $\Longrightarrow $
Shellable $\Longrightarrow$
Constructible $\Longrightarrow$
Homotopically Cohen-Macaulay
\label{thm_vd_implies_wedge}
\end{theorem}


We recall the following definitions and theorems from
discrete Morse theory.  See~\cite{Forman,Forman_2,KozlovBook}
for further details.

\begin{definition}
\emph{A \emph{partial matching} in a poset $P$ is a partial matching
in the underlying graph of the Hasse diagram of $P$, that is, a
subset $M\subseteq P\times P$ such that $(x,y)\in M$ implies
$x\prec y$ and each $x\in P$ belongs to at most one element
of $M$.  For $(x,y)\in M$ we write $x=d(y)$ and $y=u(x)$,
where $d$ and $u$ stand for down and up, respectively.}
\end{definition}

\begin{definition}
\emph{A partial matching $M$ on $P$ is \emph{acyclic} if there does not exist a
cycle
$$
z_1 \succ d(z_1) \prec z_2 \succ d(z_2) \prec \cdots \prec z_n \succ
d(z_n) \prec z_1,
$$
in $P$ with $n\geq 2$, and all $z_i\in P$ distinct.  Given a partial
matching, the unmatched
elements are called \emph{critical}.  If there are no
critical elements,  the acyclic matching is \emph{perfect}.}
\label{def_acyclic}
\end{definition}

We now state the main result from discrete Morse theory.  For a
simplicial complex $\Delta$, let
$\mathcal{F}(\Delta)$ denote the
poset of faces of $\Delta$ ordered by inclusion.

\begin{theorem}
Let $\Delta$ be a simplicial complex and let $M$ be an acyclic matching on the face poset of$~\Delta$.
Let $c_i$ denote the number of critical $i$-dimensional cells of $\Delta$.
The space $\Delta$ is homotopy equivalent to a cell complex $\Delta_c$ with $c_i$ cells of dimension $i$
for each $i\ge 0$, plus a single $0$-dimensional cell in the case where the empty set is paired in the matching. \label{cor_morse}
\end{theorem}

\begin{remark}
If the critical cells of an acyclic matching on $\Delta$ form a subcomplex $\Gamma$ of $\Delta$, then $\Delta$ simplicially collapses to $\Gamma$, implying that $\Gamma$ is a deformation retract of $\Delta$.
\end{remark}

It will be convenient for us to make use of reduced discrete Morse theory;
that is, we will include the empty set.
In particular, if the matching in Theorem~\ref{cor_morse}
is perfect, then $\Delta_c$ is contractible.  Also, if the
matching has exactly one critical cell then $\Delta_c$ is
a $d$-sphere where $d$ is
the dimension of this cell.

Given a set of critical cells of differing dimension, in general
it is impossible to conclude that the CW complex~$\Delta_c$ is
homotopy equivalent to a wedge of spheres.  See
Kozlov~\cite{Kozlov_suff_cond} for a non-trivial example.
However, when some critical cells are facets, it may be possible
to say more as seen in the following theorem.

\begin{theorem}
Let $M$ be a Morse matching on $\mathcal{F}(\Delta)$
with $k_i$ critical cells of dimension $i$.
Assume for some positive integer $j$ that there are no critical cells of dimension less than $j$ and
that all critical cells of dimension $j$ are facets.
Then the complex $\Delta$ is homotopy equivalent to the wedge
\[
X \vee \bigvee_{k_j}{\mathbb{S}^j},
\]
where $X$ is a CW complex consisting of a $0$-cell and cells of dimension greater than $j$.
\label{thm_maximal_cells}
\end{theorem}
\begin{proof}
Since there are no critical cells of dimension less than $j$, the
$(j-1)$-skeleton of $\Delta$ is contained in a collapsible sub-complex $\Delta'$ of
$\Delta$.  
As such, the quotient $\Delta / \Delta'$ is homotopy equivalent to $\Delta$
and we may obtain the CW complex $\Delta / \Delta'$ by attaching $k_j$ cells of dimension $j$ to a $0$-cell and
adding all other cells of dimension greater than $j$ to the same $0$-cell.

\end{proof}

Kozlov~\cite{Kozlov_suff_cond} gives a more general sufficient
condition on an acyclic Morse matching for the complex to be
homotopy equivalent to a wedge of spheres enumerated by the critical cells.

It is often useful to create acyclic partial matchings on several different sections of the face poset of a simplicial complex and then combine
them to form a larger acyclic partial matching on the entire poset.
This process is detailed in the following theorem known as the \textit{Cluster lemma} in \cite{JonssonBook} and as the \textit{Patchwork theorem} in \cite{KozlovBook}.

\begin{theorem}\label{patchwork}
Assume that $\varphi : P \rightarrow Q$ is an order-preserving map.
For any collection of acyclic matchings on the subposets $\varphi^{-1}(q)$ for $q\in Q$, the union of these matchings is itself an acyclic matching on $P$.
\end{theorem}

\section{Triangular Boards}

Let $\Sigma(\Pn)$ denote the simplicial complex formed by
all non-attacking rook placements on the triangular board $\Pn$.
We call a rook placement \emph{maximal} if no other rook
can be added to the placement, that is, every square
on the board is attacked.

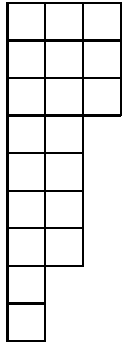
\begin{figure}

\begin{center}
\setlength{\unitlength}{5mm}
\begin{picture}(3,9)
\put(0,0){\line(0,1){9}}
\put(1,0){\line(0,1){9}}
\put(2,2){\line(0,1){7}}
\put(3,6){\line(0,1){3}}

\put(0,0){\line(1,0){1}}
\put(0,1){\line(1,0){1}}
\put(0,2){\line(1,0){2}}
\put(0,3){\line(1,0){2}}
\put(0,4){\line(1,0){2}}
\put(0,5){\line(1,0){2}}
\put(0,6){\line(1,0){3}}
\put(0,7){\line(1,0){3}}
\put(0,8){\line(1,0){3}}
\put(0,9){\line(1,0){3}}
\end{picture}
\end{center}
\caption{The Ferrers board $\Psi_{3,4,2}$.}
\label{figure_tri_board}
\end{figure}

\begin{theorem}
Let $a_i \ge i$ for all
$i=1, \ldots, n$. Then the simplicial complex
$\Sigma(\Pn)$ is vertex decomposable.
\label{Mthm}
\end{theorem}

This theorem does not extend further in general.
From a cursory glance at complexes $\Sigma(\Pn)$,
where $a_i \ge i-1$ for $i = 1, \ldots , n$ we allow for complexes such
as  $\Sigma(\Psi_{2, 1, 0})$ which is not pure.  In addition, loosening our conditions to allow $a_i \geq i-2$
for $i=1, \ldots, n-1$ and $a_n \geq n$ allows complexes such as
$\Sigma(\Psi_{4,0,0})$ which is a torus~\cite{Bjorner_et_al}.

In order to prove Theorem~\ref{Mthm}, we need the following lemma.
Recall that the squares of the first column of $\Pn$ are labeled
$(1,1), \ldots, (1,p)$ where $\displaystyle{p= \sum^n_{i=1} a_i}$.
Let $V_j$ denote a collection of the top $j$ squares of the first column,
that is, $V_j= \{(1, p-j+1), \ldots , (1,p)\}$ for $j=1, \ldots , p$ and let $V_0 = \emptyset$.

\begin{lemma}
Consider $\Pn$ with $a_i \ge i$ for all
$i=1, \ldots, n$.
Then for $j=0, \ldots, p-1$
the simplicial complex $\del_{\Sigma(\Pn)}(V_j)$ is pure
of dimension $n-1$.
\label{pure}
\end{lemma}

\begin{proof}

We have two cases to consider.

Let $j=0$.  Then del$_{\Sigma(\Pn)}(V_0) = \Sigma(\Pn)$.
Any facet of $\Sigma(\Pn)$ comes from some maximal rook placement on $\Pn$.  Any maximal rook placement must cover
the rectangular board $n \times a_n$.  Since $a_n \geq n$, this requires exactly $n$ rooks, one in each of
the $n$ columns.  Since every column
contains a rook, the entire board $\Pn$ is covered.

Let $1 \le j \le p-1$.  Here, $\del_{\Sigma(\Pn)}(V_j)$ is the simplicial complex formed
by all non-attacking rook placements on the Ferrers board $\Pn$ where
the top $j$ squares in the first column have been deleted.  Any maximal rook placement on this board
must cover the $a_n \times n-1$ rectangular sub-board created by rows
$p-a_n+1,\ldots,p$ and columns $2,3,\ldots,n$.
Since $a_n \geq n-1$, this requires exactly $n-1$ rooks to cover, one in each of the
columns $2,3,\ldots,n$.  The first column will contain at least one square (namely
$(1,1)$) that is not covered by any of these $n-1$ rooks.
Thus by placing a rook in the first column, we see that any facet of
$\del_{\Sigma(\Pn)}(V_j)$ comes from a maximal rook placement utilizing $n$ rooks.
\end{proof}

\begin{proof}[Proof of Theorem \ref{Mthm}]

We will proceed by induction on $n$, the length of the largest row.

For $n=1$, we have a $1 \times a_1$ chessboard which yields a simplicial complex
that is clearly vertex decomposable.
Now assume $\Sigma(\Pk)$ is vertex decomposable and consider $\Sigma(\Pko)$.
We maintain our labeling of the vertices using $\displaystyle{p= \sum^{k+1}_{i=1} a_i}$.
We note that $\Sigma(\Pko)$ is pure by Lemma~\ref{pure}.

We claim that the vertex corresponding to the square $(1,p)$ is a shedding vertex of $\Sigma(\Pko)$.
First, $\lk_{\Sigma(\Pko)}(1,p)$ is the set of all faces in bijection with maximal rook placements on the
Ferrers board $\Pko$
where the $p$th row (i.e., top row) and the first column of $\Psi_{a_{k+1}, \ldots , a_1}$ have been deleted.
That is,
\[
\lk_{\Sigma(\Pko)}(1,p) = \Sigma(\Psi_{a_{k+1}-1,a_k,\ldots,a_2}).
\]
Note that the largest
row is now length $k$.  Since $a_i \ge i-1$ and $a_{k+1}-1 \ge k$ the link of $(1,p)$ is vertex
decomposable by our induction hypothesis.

We now must show that $\del_{\Sigma(\Pko)}(1,p)$ is vertex decomposable by showing that the vertex corresponding to the square
$(1,p-1)$ is a shedding vertex which begins a recursive process.

At the $j$th iteration, we need to show that $\Delta_j=\del_{\Sigma(\Pko)}(V_j)$ is vertex decomposable
by showing that the vertex corresponding to $(1,p-j)$ is a shedding vertex.  Suppose row $p-j$ has length $\ell$. First,
$\lk_{\Delta_j}(1,p-j)$ corresponds to the set of all rook placements on the
Ferrers board $\Psi_{a_{k+1}, \ldots , a_1}$
with row $p-j$ and the first column deleted.
This board has $a_{i}$ rows of length $i-1$ for $i=2, \ldots, \widehat{\ell}, \ldots , k+1$
and $a_{\ell}-1$ rows of length $\ell -1$.
From this we see that
\[
\lk_{\Delta_j}(1,p-j) = \Sigma(\Psi_{a_{k+1},\ldots,a_{\ell}-1,\ldots,a_2}).
\]
Once again as $a_i \ge i-1$ and $a_{\ell}-1 \ge \ell - 1$, the link $\lk_{\Delta_j}(1,p-j)$ is
vertex decomposable by our induction hypothesis.  Similarly, whether or not the deletion $\del_{\Delta_j}(1,p-j)$ is vertex decomposable
remains undetermined and we proceed to another iteration of this process.

At the final step of this process we test
the link and deletion of the vertex corresponding to the square $(1,1)$ on the board $\Psi_{a_{k+1}, \ldots , a_2} \cup \{(1,1)\}$, where
$(1,1)$ forms its own row and column.  This complex remains pure by Lemma~\ref{pure}.  Moreover,
$\lk_{\Delta_{p-1}}(1,1)$ = $\del_{\Delta_{p-1}}(1,1) = \Sigma(\Psi_{a_{k+1}, \ldots , a_2})$
which is vertex decomposable by our induction hypothesis, verifying vertex decomposability by
moving backwards through our deletions. Thus $\Sigma(\Pko)$ is vertex decomposable.
\end{proof}

Since $\Sigma(\Pn)$ is vertex decomposable for $a_i \ge i$ for all
$i=1, \ldots, n$,
we know by Theorem~\ref{thm_vd_implies_wedge} that
it will be homotopy equivalent to a wedge of spheres of dimension $n-1$ or contractible.  The number of spheres
can be computed by finding the reduced Euler characteristic which is the alternating sum of the coefficients of the
$f$-vector.  Let $\ell(i)$ denote the length of column $i$ in $\Pn$ with
$\displaystyle{\ell(i) = \sum_{j=i}^{n} a_j}.$

\begin{theorem}
The coefficients of the $f$-vector of $\Sigma(\Pn)$ are given by
$$
f_i=\sum_{S\in\binom{[n]}{i+1}} \prod_{j=0}^{i} (\ell(s_j)-j),
$$
where $S=\{s_0 > s_1 > \cdots > s_{i}\}$.
\end{theorem}
\begin{proof}
We will compute the coefficients of the $f$-vector by considering all rook placements on the board $\Pn$.
Let $S$ be the collection of $i+1$ columns
where the rooks are placed.  Notice that $\ell(s_k) \leq \ell(s_{k+1})$.  Therefore, placing a rook
in column $s_k$ removes a possible location to place a rook in columns $s_{k+1},\ldots, s_{i}$.
Thus, there are $\displaystyle{\prod_{j=0}^{i} (\ell(s_j)-j)}$ ways to place $i+1$ rooks on
these $i+1$ columns.  The result follows by summing over all subsets of $i+1$ columns.
\end{proof}

Theorem~\ref{Mthm} mirrors the result of Ziegler~\cite{Ziegler} which
says that the (rectangular) chessboard complex $M_{n,m}$ is vertex decomposable
if $n\ge 2m-1$.  That is, extending a triangular board, like extending a
rectangular board, far enough allows one to conclude that the associated complex
is vertex decomposable.

\section{Facet Enumeration}\label{stirling_facets}

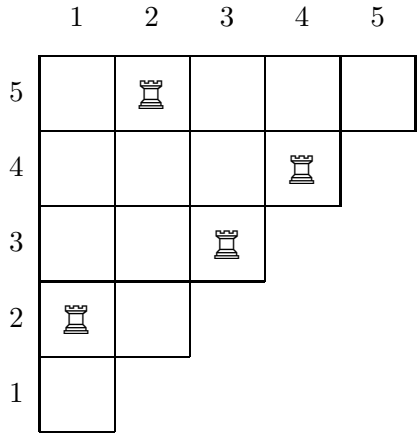
\begin{figure}
\begin{center}
\setlength{\unitlength}{1cm}
\begin{picture}(6,6)
\put(0,0){\line(0,1){5}}
\put(1,0){\line(0,1){5}}
\put(2,1){\line(0,1){4}}
\put(3,2){\line(0,1){3}}
\put(4,3){\line(0,1){2}}
\put(5,4){\line(0,1){1}}
\put(0,5){\line(1,0){5}}
\put(0,4){\line(1,0){5}}
\put(0,3){\line(1,0){4}}
\put(0,2){\line(1,0){3}}
\put(0,1){\line(1,0){2}}
\put(0,0){\line(1,0){1}}

\put(3.3,3.3){\Large{\rook}}
\put(2.3,2.3){\Large{\rook}}
\put(.3,1.3){\Large{\rook}}
\put(1.3,4.3){\Large{\rook}}

\put(-.4,.4){1}
\put(-.4,1.4){2}
\put(-.4,2.4){3}
\put(-.4,3.4){4}
\put(-.4,4.4){5}

\put(.4,5,4){1}
\put(1.4,5,4){2}
\put(2.4,5,4){3}
\put(3.4,5,4){4}
\put(4.4,5,4){5}
\end{picture}
\end{center}
\caption[A rook placement in bijective correspondence with \{\{1,3,4,5\},\{2,6\}\}.]{A rook placement in bijective correspondence with the partition
\{\{1,3,4,5\},\{2,6\}\}.}
\label{figure_rook_placement}
\end{figure}

We now turn our attention to the Stirling complex. Recall
the Stirling complex $\St(n)$ is the simplicial
complex associated to valid rook placements on the
triangular board of size $n$, $\Psi_{1,\ldots,1}$.  It is
clear that the Stirling complex is not pure. In this section,
we will enumerate the facets of the Stirling complex in each
dimension.

The $f$-vector of the Stirling complex is given by
Stirling numbers of the second kind; that is,
faces of the Stirling complex are in bijection with
partitions. This is done using the map $\mathcal{R}$
where any placement of $k$ non-attacking rooks gets mapped
to the partition $A$ where if a rook occupies the square
$(i,j)$ then $i$ and $j+1$ are in the same block of the
partition~$A$, see~\cite[Corollary~2.4.2]{StanleyVol1}.
In what follows, we show that
facets of the Stirling complex are in bijection with
a particular subset of partitions.

\begin{definition}
\emph{Let $B$ and $C$ be two disjoint nonempty subsets of $[n]$.  Then $B$ and $C$ are
\emph{intertwined} if $\max(B) > \min(C)$ and $\max(C) > \min(B)$.  We say a partition $P$
is \emph{intertwined} if every pair of blocks in $P$ is intertwined.}
\end{definition}

The idea of intertwined partitions first appeared with the
use of the \emph{intertwining number} of a partition
in~\cite{Ehrenborg_Readdy} where they were used to provide a
combinatorial interpretation for $q$-Stirling numbers of the second kind.
The following definitions are from~\cite{Ehrenborg_Readdy}.
For two integers $i$ and $j$, define the interval
$\inte(i,j)$ to be the set
$$
\inte(i,j) = \{n\in \BZ : \min(i,j) < n < \max(i,j)\}.
$$
\begin{definition}
\emph{For two disjoint nonempty subsets $B$ and $C$ of $[n]$, define
the \emph{intertwining number} $\iota(B,C)$ to be
$$
\iota(B,C) = |\{(b,c)\in B\times C:\inte(b,c)\cap(B\cup C)=\emptyset\}|.
$$}
\end{definition}

These two ideas are connected as can be seen in the following proposition.

\begin{proposition}
Let $B$ and $C$ be two disjoint nonempty subsets of $[n]$.  Then $B$ and $C$ are
intertwined if and only if $\iota(B,C)\geq 2$.
\label{prop_max_min_intertwined}
\end{proposition}
\begin{proof}
($\Rightarrow$) Suppose $\min(B) < \min(C) < \max(B)$.
Let $b_0$ be the maximum element of $B$ such that
$\min(B)\leq b_0 < \min(C)$.  Then
$\inte(b_0,\min(C))\cap(B\cup C) = \emptyset$.
Let $c_1$ be the maximum element of $C$ such
that $c_1 < \max(B)$ and $b_1$
be the minimum element of $B$ such that
$c_1 < b_1$.  Then
$\inte(b_1,c_1)\cap(B\cup C) = \emptyset$.  Therefore,
the intertwining number is at least 2.

($\Leftarrow$) Suppose $\max(B)<\min(C)$.
Clearly, for $b\in B$ and $c\in C$,
$\inte(b,c)\cap(B\cup C)\neq \emptyset$
unless $b=\max(B)$ and $c=\min(C)$.  Thus, the
intertwining number is 1.
\end{proof}

The bijection between facets of the Stirling complex and
intertwined partitions can now be verified.

\begin{theorem}
The set of maximal rook placements with $k$ rooks on a triangular board $\Psi_{1,\ldots,1}$
of size $n$ is in bijection with intertwined partitions of $[n+1]$ into $n+1-k$ blocks.
\label{intertwined}
\end{theorem}
\begin{proof}
We first show that a maximal rook placement gives rise to an intertwined partition.
Let $P = \{P_1,P_2, \ldots, P_{n+1-k}\}$ be a partition of $n+1$ into $n+1-k$ blocks
such that there exist two blocks, $P_i$ and $P_j$, that are not intertwined.
Then, without loss of generality,
$M_i < m_j$ where $M_i$ is the maximal element of $P_i$ and $m_j$ is the minimal
element of $P_j$.  It is clear that $\mathcal{R}^{-1}(P) \cup \{(M_i, m_{j-1})\}$ is a
rook placement.  This shows $\mathcal{R}^{-1}(P)$ is not maximal.




We now show that an intertwined partition gives rise to a maximal rook placement.
Suppose~$R_k$ is not a maximal rook placement.  We consider $\CR(R_k) = Q$.
As $R_k$ is not maximal, there exists a square $(i,j)$ where we may place a rook.
As there is no rook
in column $i$, this implies that $i$ is the maximal element of some block
$Q_i$ in $Q$.  Similarly, no rook in row $j$ implies that $j+1$
is the minimal element of some block $Q_j$ in $Q$.  Hence, $Q$ contains
two blocks that are not intertwined.
\end{proof}

We now count the number of partitions with intertwined blocks.

\begin{theorem}
The number of partitions of $[n]$ into $k$ intertwined blocks is given by
$$
I(n,k)=(k-1)! \sum_{i=k-1}^{n-k}S(i,k-1)\cdot S(n-i,k).
$$
\label{partition_count}
\end{theorem}
\begin{proof}
Let $1\le i \le n$.  Let $P$ be a partition of
$[i]$ into $k-1$ blocks.  This can be done in $S(i,k-1)$ ways.  Next, let
$Q$ be a partition of the remaining $n-i$ elements into $k$ blocks which can
be done in $S(n-i,k)$ ways.
In order to combine these two partitions into a single partition
of $[n]$ into $k$ intertwined blocks, ignore the block containing $i+1$
in $Q$ and pair each block of $P$ to exactly one unique block in $Q$.  This
can be done in $(k-1)!$ ways.
Clearly any partition obtained with this construction is intertwined.
It is also straightforward to see that any intertwined partition can be
obtained in this way.

We now show that every partition with intertwined blocks can
be obtained uniquely in this way.
When summing over $i$, notice that $i+1$ is the largest minimal element
of the blocks of the partition.  Thus, as $i$ varies, so do the intertwined partitions
generated.
For a fixed $i$, it is clear that any intertwined partition can be obtained in at most one way.
For an example of this process, see Figure~\ref{fig_intertwined}.

\end{proof}

\begin{figure}
\setlength{\unitlength}{0.75mm}
\begin{center}
\begin{picture}(110,50)(0,0)
\thicklines

\put(10,15){\circle*{2}}
\put(9,10){\small{1}}
\put(20,15){\circle*{2}}
\put(19,10){\small{2}}
\put(30,15){\circle*{2}}
\put(29,10){\small{3}}

\put(40,0){\line(0,1){50}}
\put(50,0){\line(0,1){50}}

\put(60,15){\circle*{2}}
\put(59,10){\small{4}}
\put(70,15){\circle*{2}}
\put(69,10){\small{5}}
\put(80,15){\circle*{2}}
\put(79,10){\small{6}}
\put(90,15){\circle*{2}}
\put(89,10){\small{7}}
\put(100,15){\circle*{2}}
\put(99,10){\small{8}}


\setdashes<2pt>

\setlinear
\plot 85 53.5 107 53.5 /
\plot 85 64 107 64 /

\setdots<2pt>
\plot 85 64 107 53.5 /
\plot 85 53.5 107 64 /

\setsolid


\thicklines
\qbezier(10,15)(10,30)(40,30)

\qbezier(20,15)(25,30)(30,15)
\qbezier(30,15)(30,25)(40,25)

\qbezier(60,15)(80,50)(100,15)

\qbezier(50,25)(70,25)(70,15)
\qbezier(70,15)(80,30)(90,15)

\qbezier(50,30)(80,30)(80,15)
\end{picture}
\end{center}
\caption[This diagram denotes the intertwined partition $\{\{1,6\},\{2,3,5,7\},\{4,8\}\}$ ]{This diagram denotes the intertwined partition
$\{\{1,6\},\{2,3,5,7\},\{4,8\}\}$
or the intertwined partition
$\{\{1,5,7\},\{2,3,6\},\{4,8\}\}$
depending on how the edges from the first part
are connected to the edges of the second part.
}
\label{fig_intertwined}
\end{figure}
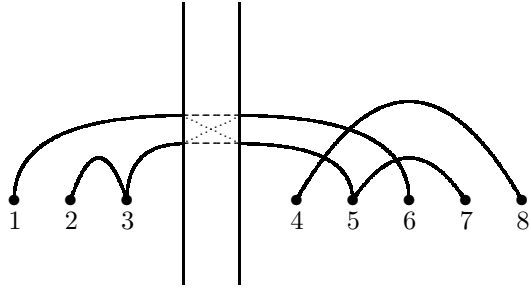

\begin{corollary}
The number of maximal rook placements of size $n-k$ on a triangular
board $\Psi_{1,\ldots,1}$ of size $n$ is given by
$$
F_{n-k}^n = k! \sum_{i=k}^{n-k} S(i,k)\cdot S(n+1-i,k+1).
$$
\label{cor_facet_count}
\end{corollary}

Using this corollary, we can write the generating function for the facets
of the Stirling complex.  It is interesting to note that the generating function
obtained is the product of the generating function for $S(n,k)$, a shift of the
generating function for $S(n,k+1)$, and $k!$.

\begin{corollary}
The generating function for $F_{n-k}^n$ is given by
$$
\sum_{n\geq 0} F_{n-k}^n x^n
=
\frac{k!\cdot x^{2k}}
     {\left(\prod_{i=1}^k (1-ix)\right)^2\cdot(1-(k+1)x)}.
$$
\end{corollary}
\begin{proof}
We have
\begin{eqnarray*}
\sum_{n\geq 0} F_{n-k}^n x^n
&=& k! \sum_{n\geq 0} \sum_{i=0}^{n} S(i,k)\cdot S(n+1-i,k+1) x^n\\
&=& k! \left( \sum_{n\geq 0} S(n,k)x^n \right)\cdot
       \left( \sum_{n\geq 0} S(n+1,k+1)x^n \right)\\
&=& k! \left( \frac{x^k}{\prod_{i=1}^k (1-ix)} \right)\cdot
       \left( \frac{x^k}{\prod_{i=1}^{k+1} (1-ix)} \right)\\
&=& \frac{k!\cdot x^{2k}}
         {\left(\prod_{i=1}^k (1-ix)\right)^2\cdot(1-(k+1)x)}.
\end{eqnarray*}
\end{proof}

\section{Homology}\label{stirling_homology}

In this section we examine the topology of the Stirling complex.  Work on this
has been done by Barmak~\cite{Barmak} where he showed that the Stirling
complex $\St(n)$ is $\left\lfloor \frac{n-3}{2} \right\rfloor$-connected.
Our technique uses discrete
Morse theory by defining poset maps and creating a Morse matching using
the Patchwork theorem.  We provide an alternate proof of Barmak's connectivity
bound and further give a partial
description of its homotopy type when $n$ is even.

\begin{figure}
\begin{center}
{
\setlength{\unitlength}{1cm}
\begin{picture}(6,6)
\put(0,0){\line(0,1){5}}
\put(1,0){\line(0,1){5}}
\put(2,1){\line(0,1){4}}
\put(3,2){\line(0,1){3}}
\put(4,3){\line(0,1){2}}
\put(5,4){\line(0,1){1}}
\put(0,5){\line(1,0){5}}
\put(0,4){\line(1,0){5}}
\put(0,3){\line(1,0){4}}
\put(0,2){\line(1,0){3}}
\put(0,1){\line(1,0){2}}
\put(0,0){\line(1,0){1}}

\put(4.4,4.4){$1$}
\put(.4,.4){$5$}
\put(3.4,3.4){$2$}
\put(1.4,1.4){$4$}
\put(2.4,2.4){$3$}
\put(2.4,3.4){$7$}
\put(1.4,2.4){$8$}
\put(3.4,4.4){$6$}
\put(.4,1.4){$9$}
\put(2.3,4.4){$10$}
\put(.3,2.4){$12$}
\put(1.3,3.4){$11$}
\put(1.3,4.4){$13$}
\put(.3,3.4){$14$}
\put(.3,4.4){$15$}

\put(-.4,.4){1}
\put(-.4,1.4){2}
\put(-.4,2.4){3}
\put(-.4,3.4){4}
\put(-.4,4.4){5}

\put(.4,5,4){1}
\put(1.4,5,4){2}
\put(2.4,5,4){3}
\put(3.4,5,4){4}
\put(4.4,5,4){5}

\end{picture}

}

\end{center}
\caption{The ordering $Q$ on the squares of the triangular board $\Psi_{1,1,1,1,1}$}
\label{new_order}
\end{figure}

Fix a postive integer $n$ and consider the corresponding
triangular board $\Psi_{1,\ldots, 1}$.  
We define a total ordering $Q$ on the squares of this board  as follows:
for $(i,j), (k,l) \in \Psi_{1,\ldots, 1}$, we have $(i,j) <_Q (k,l)$ if either $|j-i| < |k-l|$, or $|j-i| = |k-l|$ and $i>k$.
Using this order, we may enumerate the squares of $\Psi_{1,\ldots, 1}$ one through $\binom{n+1}{2}$, see Figure~\ref{new_order}.

Let $Q_1$ be the sub-chain of $Q$ consisting of the lowest elements
$(n,n)<_Q (n-1,n-1)<_Q \cdots <_Q (1,1)$ 
adjoined with a maximal element $\widehat{1}_{Q_1}$.
We define a map $\varphi$ from the face poset of $\St(n)$
to the poset $Q_1$.  For $x\in \St(n)$, let
\[
\varphi(x) =
\left\{
\begin{array}{ll}
(i,i), & \mbox{if\ }(i,i) \mbox{\ is the smallest in $Q_1$ such that}\\
       & x\cup \{(i,i)\} \in \St(n) \\                         
\widehat{1}_{Q_1}, &  x\cup \{(i,i)\} \notin \St(n) \mbox{\ for every\ } i.                                              
\end{array}
\right.
\]

\begin{lemma}
The map $\varphi:\mathcal{F}(\St(n))\longrightarrow Q_1$ is an order-preserving poset map.
\end{lemma}
\begin{proof}
Let $x,y\in \St(n)$ with $x\subset y$.  Suppose $\varphi(x)=(i,i)$.  That is, $(i,i)$ is
the smallest ordered pair such that
$(1,i),(2,i),\ldots,(i-1,i),(i,i+1),(i,i+2),\ldots,(i,n)$ are not elements of $x$.  Since $y$
contains $x$, $\varphi(y)$ can be no smaller than $(i,i)$.  Therefore,
$\varphi(x)\leq \varphi(y)$.  Suppose $\varphi(x)=\widehat{1}_{Q_1}$.  Since
$y$ contains $x$, we have $\varphi(y)=\widehat{1}_{Q_1}$ also.
\end{proof}

\begin{lemma}
For $(i,i) <_{Q_1} \widehat{1}_{Q_1}$, the collection
$\{(x,x\cup \{(i,i)\})\ :\ (i,i)\not\in x\in \varphi^{-1}(i,i)\}$
is a perfect acyclic matching on the fiber $\varphi^{-1}(i,i)$.
\label{lemma_fibers}
\end{lemma}
\begin{proof}
Suppose $\varphi(x)=(i,i)$ and $(i,i)\not\in x$.
By definition of the function $\varphi$,
$u(x)=x\cup \{(i,i)\}$ is a valid rook placement
in $\St(n)$.  Also, $\varphi(u(x))= (i,i)$.
Suppose $\varphi(x)=(i,i)$ and $(i,i)\in x$.  It is clear that
$d(x)=x-\{(i,i)\}$ is a valid rook placement.  Also, removing
the element $(i,i)$ will not affect the mapping under $\varphi$.
Therefore, $\varphi(d(x))=(i,i)$.
Finally, this matching is clearly acyclic since the same element
is either added to or removed from a placement.
\end{proof}

Using the Patchwork theorem, we have an acyclic matching on
$\mathcal{F}(\St(n))$ whose only critical cells are the elements
of the fiber $\Gamma=\varphi^{-1}(\widehat{1}_{Q_1})$.  From the
definition of the function $\varphi$, the following is clear.
\begin{lemma}
The rook placement $x$ is an element of
$\Gamma$ if and only if for each $i\in [n]$, there is a $j\neq i$ such that
$(i,j)\in x$ or $(j,i)\in x$.  Additionally, the rook placement $x$ has
at least $\lceil n/2\rceil$ rooks.
\label{lemma_elements_of_gamma}
\end{lemma}

We now turn our attention to boards of even size, that is, consider $\St(2n)$.
Recall that the \emph{Durfee square} of a Ferrers board
is the largest square sub-board consisting of adjacent rows and columns and containing the upper left-hand corner of the board.
For example, the Durfee square of the board in Figure~\ref{figure_rook_placement}
is of size three.  It is clear that the Durfee square of the board $\Psi_{1,\ldots,1}$ of size $2n$
is of size $n$.

From Lemma~\ref{lemma_elements_of_gamma}
we know that for $x\in \Gamma$, $x$ will contain at least $n$ rooks.  Assume that $x\in \St(2n)$
contains exactly $n$ rooks.  The placement $x$ falls into one of two cases.

\noindent {\bf Case 1:} Assume the rooks of $x$ are completely contained in the Durfee square.  It is clear that the
placement $x$ is maximal, and there are $n!$ such placements.  Also, since there are no rook placements
in $\Gamma$ of size $n-1$, $x$ can never be a part of a Morse matching on $\Gamma$.  Therefore, with any matching,
$x$ will be critical.

Before proceeding to the next case, we prove the following lemma.

\begin{lemma}
Let $x$ be a non-maximal rook placement of size $n$ in $\Gamma$ on the board $\Psi_{1,\ldots,1}$ of size $2n$.  If
position $(i,j)$ is unattacked in $x$, then $x\cup (i,j)$ is an element of $\Gamma$.  Additionally, $x$ is the only
element of $\Gamma$ properly contained in $x\cup (i,j)$.
\label{lemma_non_max_placements}
\end{lemma}
\begin{proof}
It is clear from the definition of $\Gamma$ that if $x\in \Gamma$ then $x\cup (i,j) \in \Gamma$.
Since $x$ is an element of $\Gamma$ with only $n$ rooks and Lemma~\ref{lemma_elements_of_gamma} indicates that all
of the $2n$ diagonal positions must be attacked, each rook of $x$ uniquely attacks two diagonal positions.  Assume
that $x'= x\cup (i,j) - (k,\ell)$ for some $(k,\ell) \neq (i,j)$.  Since $(i,j)$ cannot attack both positions $(k,k)$
and $(\ell, \ell)$, one of those diagonal positions is unattacked in $x'$.  Therefore, $x'$ is not an element of $\Gamma$.
\end{proof}

\noindent {\bf Case 2:} Assume that at least one of the $n$ rooks of $x$ is outside the Durfee square.  From
Corollary~\ref{cor_facet_count}, we know that the number of maximal rook placements of size $n$ in
the board $\Psi_{1,\ldots,1}$ of size $2n$ is $n!$.  Since there are $n!$ maximal rook placements of size $n$
contained in the Durfee square, $x$ cannot be maximal.  Therefore, there is a square $(i,j)$ that is not attacked.
Thus, by Lemma~\ref{lemma_non_max_placements},
$x\cup (i,j)$ is also in $\Gamma$, and $x\cup(i,j)$ covers no other element of $\Gamma$.

Define a matching of $\Gamma$ with the pairs $(x,x\cup(i,j))$ where $x$ is a rook placement from case 2.  Since
$x\cup(i,j)$ covers no other elements of $\Gamma$, the matching is clearly acyclic. This is, therefore, a Morse matching
where there are only $n!$ critical cells of dimension $n-1$.

Using Theorem~\ref{thm_maximal_cells} it is now straightforward to see the following about the topology of the Stirling complex.
\begin{theorem}
The Stirling complex $\St(n)$ is
homotopy equivalent to a CW complex with one $0$-cell and no
cells of dimension $k$ for $0<k<\lceil n/2 \rceil-1$
and for $k\geq n-1$.
Moreover, the Stirling complex $\St(2n)$ is homotopy equivalent to a wedge of
$n!$ spheres of dimension $n-1$ with a space $X$ where $X$ is $(n-1)$-connected.
\label{thm_spheres_x}
\label{thm_zero_betti}
\end{theorem}

We now discuss two corollaries.

\begin{corollary}
The Stirling complex $\St(n)$ is exactly $\left\lfloor\frac{n-3}{2}\right\rfloor$-connected.
\label{Stirling_connected}
\end{corollary}
\begin{proof}
When $n$ is even, this follows immediately from Theorem~\ref{thm_spheres_x}.  When $n$ is odd, note that the facet
corresponding to the placement
$\displaystyle{ \left\{\left(1,\frac{n+1}{2}\right),\left(2,\frac{n+1}{2}+1\right),\ldots, \left(\frac{n+1}{2},n\right)\right\}}$
is critical while its boundary is contained in a contractible sub-complex $K$, which is the union of the fibers
$\varphi^{-1}(i,i)$ of the map in Lemma~\ref{lemma_fibers}.  Therefore, $\St(n)$ is homotopy equivalent to $\St(n) / K$ which
is a sphere, $S^{\frac{n-1}{2}}$, wedged with a CW complex consisting of a $0$-cell and no other cells with dimension smaller than
$\frac{n-1}{2}$. 

\end{proof}


We note that Corollary~\ref{Stirling_connected} can also be obtained from work done by Barmak~\cite{Barmak}
by observing that the star cluster of the diagonal
contains the boundary of a facet $\sigma$ of dimension $\lfloor\frac{n-1}{2} \rfloor$, but does not contain
$\sigma$ itself.

Moreover, we obtain the following corollary from Theorem~\ref{thm_spheres_x}.

\begin{corollary}
The $(n-1)$st reduced Betti number of the Stirling complex $\St(2n)$ is $n!$.
\end{corollary}

\section{Further questions}

\begin{figure}
\begin{center}
\subfigure[The order $P$ for $n=5$.]
{
\setlength{\unitlength}{6mm}
\begin{picture}(15,8)(0,0)
\put(5,0){\small{$10$}}
\put(5,2){\small{$2$}}
\put(5,4){\small{$8$}}
\put(5,6){\small{$4$}}
\put(5,8){\small{$6$}}
\put(6,0){\circle*{0.3}}
\put(6,2){\circle*{0.3}}
\put(6,4){\circle*{0.3}}
\put(6,6){\circle*{0.3}}
\put(6,8){\circle*{0.3}}

\put(10,0){\small{$7$}}
\put(10,2){\small{$5$}}
\put(10,4){\small{$9$}}
\put(10,6){\small{$3$}}
\put(9,0){\circle*{0.3}}
\put(9,2){\circle*{0.3}}
\put(9,4){\circle*{0.3}}
\put(9,6){\circle*{0.3}}

\put(6,0){\line(0,1){8}}
\put(9,0){\line(0,1){6}}
\end{picture}
\label{fig_order_p}
}
\subfigure[The order $Q'$ of the squares of a triangular board
of size five.]
{
\setlength{\unitlength}{1cm}
\begin{picture}(6,6)
\put(0,0){\line(0,1){5}}
\put(1,0){\line(0,1){5}}
\put(2,1){\line(0,1){4}}
\put(3,2){\line(0,1){3}}
\put(4,3){\line(0,1){2}}
\put(5,4){\line(0,1){1}}
\put(0,5){\line(1,0){5}}
\put(0,4){\line(1,0){5}}
\put(0,3){\line(1,0){4}}
\put(0,2){\line(1,0){3}}
\put(0,1){\line(1,0){2}}
\put(0,0){\line(1,0){1}}

\put(4.4,4.4){$1$}
\put(.4,.4){$2$}
\put(3.4,3.4){$3$}
\put(1.4,1.4){$4$}
\put(2.4,2.4){$5$}
\put(2.4,3.4){$6$}
\put(1.4,2.4){$7$}
\put(3.4,4.4){$8$}
\put(.4,1.4){$9$}
\put(2.3,4.4){$10$}
\put(.3,2.4){$11$}
\put(1.3,3.4){$12$}
\put(1.3,4.4){$13$}
\put(.3,3.4){$14$}
\put(.3,4.4){$15$}

\put(-.4,.4){1}
\put(-.4,1.4){2}
\put(-.4,2.4){3}
\put(-.4,3.4){4}
\put(-.4,4.4){5}

\put(.4,5,4){1}
\put(1.4,5,4){2}
\put(2.4,5,4){3}
\put(3.4,5,4){4}
\put(4.4,5,4){5}
\end{picture}
\label{fig_stirling_order}
}
\end{center}
\caption{}
\end{figure}

\begin{enumerate}
\item 
For a positive integer $n$, let $P$ be the following poset
on the set $\{2,3,4,\ldots,2n\}$.
The even integers have the order
\begin{eqnarray*}
 2n <_P 2(1) <_P 2(n-1) <_P 2(2) &<_P& \cdots  <_P  2(n-k) <_P 2(k+1)   \\
&<_P& 2(n-k-1) <_P \cdots <_P 2\lceil n/2 \rceil ,
\end{eqnarray*}
while the odd integers have the cover relations $k_i <_P k_{i+1}$ where
$k_{i+1} = k_i + 2i\cdot(-1)^{n+i+1}$ and
$k_1 = 2\lceil n/2 \rceil + 1$.  The evens and odds are not comparable,
see Figure~\ref{fig_order_p}.

Using $P$, we define a total order $Q'$ on the squares of the triangular board
$\Psi_{1,\ldots,1}$.
For $(i,j)\in [n]\times[n]$ with $i\leq j$, $(i,j)<_{Q'} (k,\ell)$ if
$j-i < \ell-k$ in the standard order.  If $j-i=\ell -k$ then
$(i,j)<_{Q'} (k,\ell)$ if $i+j <_P k+\ell$.

Informally, this order is
obtained by starting on the largest diagonal and alternating upper-right
to lower-left from the outside to the middle. We continue on the next
diagonal alternating from the middle to the outside. The next diagonal
moves again from the outside to the middle, etc., see
Figure~\ref{fig_stirling_order}.    

Consider the following matching, $M$,
of the rook placements on this board.  Begin by pairing placements $x \subset y$ if
the set difference $y-x$ is square one in the order $Q'$.  Of the remaining placements, match $x \subset y$
if the set difference $y-x$ is square two in the order $Q'$, and continue this process until there are no more
possible matches.

Using the program Macaulay2, we have been able to compute the reduced homology of the
Stirling complex up to $\St(8)$, see Table~\ref{tbl_stirling_betti}.  The
unmatched placements from the matching $M$ are in line with the number and dimension
of cells in the CW complex description of
these first eight complexes.  Is $M$ a Morse matching?  Is this matching maximal?  If so, is there a way
to count the critical cells?

\begin{table}
\begin{center}
\begin{tabular}{c|cccccc}
$n$ & $\widetilde{\beta}_0$ & $\widetilde{\beta}_1$
  & $\widetilde{\beta}_2$ & $\widetilde{\beta}_3$
  & $\widetilde{\beta}_4$ & $\widetilde{\beta}_5$\\
\hline
1 & 0\\
2 & 1\\
3 & 0 & 1\\
4 & 0 & 2\\
5 & 0 & 0 & 9\\
6 & 0 & 0 & 6 & 15\\
7 & 0 & 0 & 0 & 58 & 8\\
8 & 0 & 0 & 0 & 24 & 292 & 1\\
\end{tabular}
\end{center}
\caption{The reduced Betti numbers of the Stirling complex through $\St(8)$.}
\label{tbl_stirling_betti}
\end{table}

\item Can anything more be said about the homotopy type of $\St(n)$?  In particular,
what is the complex $X$ from Theorem~\ref{thm_spheres_x}?  The numerical computations
of homology died at $\St(9)$. Coincidentally, this is the first instance where
the Durfee square of the board has size $5$.  Shareshian and Wachs~\cite{Shar_Wach}
showed that this was the first square board whose chessboard complex contained
torsion in its bottom non-vanishing homology.  Is there torsion in $\St(9)$?

\end{enumerate}

\section*{Acknowledgements}
The authors would like to thank Benjamin Braun and Richard Ehrenborg
for reading earlier versions of this paper as well as David Cook
for his help with Macaulay2.
The authors would also like to thank the anonymous referees for many helpful comments, suggestions,
and insights, particularly those relating to the work of Barmak and the proofs of
Theorem~\ref{thm_spheres_x} and Corollary~\ref{Stirling_connected}.
The first author was partially funded by a graduate fellowship from
National Science Foundation grant DMS-0902063.
The second author was partially funded by a graduate fellowship from
National Science Foundation grant DMS-0758321.
\newcommand{\journal}[6]{{\sc #1,} #2, {\it #3} {\bf #4} (#5), #6.}
\newcommand{\bk}[4]{\textsc{#1,} ``#2,'' #3, #4.}
\newcommand{\preprint}[3]{{\sc #1,} #2, preprint #3.}

\bigskip

\noindent
{\em E.\ Clark,
Department of Mathematics,
Western New England University,
Springfield, MA 01119,}
\{{\tt eric.clark@wne.edu}\}

\bigskip

\noindent
{\em M.\ Zeckner,
Department of Mathematics,
Adrian College,
Adrian, MI 49221,}
\{{\tt mzeckner@adrian.edu}\}


\begin{thebibliography}{10}

\bibitem{Barmak}
J.~A. {Barmak}.
\newblock {Star clusters in independence complexes of graphs}.
\newblock {\em ArXiv e-prints}, July 2010.

\bibitem{Bjorner_et_al}
A.~Bj{\"o}rner, L.~Lov{\'a}sz, S.~T. Vre{\'c}ica, and R.~T.
  {\v{Z}}ivaljevi{\'c}.
\newblock Chessboard complexes and matching complexes.
\newblock {\em J. London Math. Soc. (2)}, 49(1):25--39, 1994.

\bibitem{Cook}
D.~{Cook}, II.
\newblock {SimplicialDecomposability}.
\newblock {\em The Journal of Software for Algebra and Geometry}, (2):20--23,
  2010.

\bibitem{EhrenborgHetyei}
Richard Ehrenborg and G{\'a}bor Hetyei.
\newblock The topology of the independence complex.
\newblock {\em European J. Combin.}, 27(6):906--923, 2006.

\bibitem{Ehrenborg_Readdy}
Richard Ehrenborg and Margaret Readdy.
\newblock Juggling and applications to {$q$}-analogues.
\newblock In {\em Proceedings of the 6th {C}onference on {F}ormal {P}ower
  {S}eries and {A}lgebraic {C}ombinatorics ({N}ew {B}runswick, {NJ}, 1994)},
  volume 157, pages 107--125, 1996.

\bibitem{Forman}
Robin Forman.
\newblock Morse theory for cell complexes.
\newblock {\em Adv. Math.}, 134(1):90--145, 1998.

\bibitem{Forman_2}
Robin Forman.
\newblock A user's guide to discrete {M}orse theory.
\newblock {\em S\'em. Lothar. Combin.}, 48:Art.\ B48c, 35, 2002.

\bibitem{Garst}
Peter~Freedman Garst.
\newblock {\em Cohen-Macaulay complexes and group actions}.
\newblock ProQuest LLC, Ann Arbor, MI, 1979.
\newblock Thesis (Ph.D.)--The University of Wisconsin - Madison.

\bibitem{JonssonBook}
Jakob Jonsson.
\newblock {\em Simplicial complexes of graphs}, volume 1928 of {\em Lecture
  Notes in Mathematics}.
\newblock Springer-Verlag, Berlin, 2008.

\bibitem{KozlovBook}
Dmitry Kozlov.
\newblock {\em Combinatorial algebraic topology}, volume~21 of {\em Algorithms
  and Computation in Mathematics}.
\newblock Springer, Berlin, 2008.

\bibitem{Kozlov_suff_cond}
Dmitry Kozlov.
\newblock Discrete morse theory and hopf bundles.
\newblock {\em Pacific Journal of Mathematics}, 249(2):371--–376, 2011.

\bibitem{Shar_Wach}
John Shareshian and Michelle~L. Wachs.
\newblock Torsion in the matching complex and chessboard complex.
\newblock {\em Adv. Math.}, 212(2):525--570, 2007.

\bibitem{StanleyVol1}
Richard~P. Stanley.
\newblock {\em Enumerative combinatorics. {V}ol. 1}, volume~49 of {\em
  Cambridge Studies in Advanced Mathematics}.
\newblock Cambridge University Press, Cambridge, 1997.
\newblock With a foreword by Gian-Carlo Rota, Corrected reprint of the 1986
  original.

\bibitem{Ziegler}
G{\"u}nter Ziegler.
\newblock Shellability of chessboard complexes.
\newblock {\em Israel Journal of Mathematics}, 87:97--110, 1994.
\newblock 10.1007/BF02772986.

\end{thebibliography}
\end{document}